\documentclass{article}

\usepackage{amsmath}
\usepackage{todonotes}
\usepackage{graphicx}
\usepackage{url}
\usepackage{dsfont}
\usepackage{xargs}
\usepackage{amssymb}
\usepackage{hyperref}

\renewcommand{\vec}{\mathbf}
\newcommand{\eps}{\varepsilon}
\newcommand{\set}[1]{\left\{#1\right\}}
\newcommand{\abs}[1]{\left|#1\right|}

\newcommand{\infnorm}[1]{\left\|#1\right\|_{\infty}}
\newcommand{\R}{\mathds{R}}
\newcommand{\N}{\mathds{N}}
\newcommand{\UF}{\mathcal{U}}
\newcommand{\lexleq}{\leq_{\text{lex}}}
\newcommand{\lexlt}{\leq_{\text{lex}}}

\newcommand{\hr}[1]{\hat{#1}}  
\newcommand{\opt}[1]{{#1^*}} 

\newcommand{\F}{\mathcal{F}_{a,b}}
\newcommand{\HR}{{^*\R}}
\newcommand{\simplex}{\triangle}

\newtheorem{example}{Example}
\newtheorem{thm}{Theorem}[section]{\bfseries}{\rmfamily}
\newtheorem{lemma}[thm]{Lemma}{\bfseries}{\rmfamily}
\newtheorem{proposition}[thm]{Proposition}{\bfseries}{\rmfamily}
{\bfseries}{\rmfamily}
\newtheorem{definition}[thm]{Definition}{\bfseries}{\rmfamily}
{\bfseries}{\rmfamily}
\newtheorem{remark}[thm]{Remark}{\itshape}{\rmfamily}
\newenvironmentx{proof}[1][1=\empty]{\ifthenelse{\equal{#1}{\empty}}{\emph{Proof.}~}{\emph{Proof (#1).~}}}{\hfill$\square$\\[0.1\baselineskip]}

\title{On Game Theory Using Stochastic Tail Orders
}

\newcommand{\email}{\texttt}

\author{Stefan Rass\thanks{Institute for Artificial Intelligence and Cybersecurity, Universitaet Klagenfurt, Klagenfurt, Austria
		(\email{stefan.rass@aau.at}).}
	\and Sandra König\thanks{Austrian Institute of Technology, Center for Digital Safety \& Security, Vienna, Austria (\email{sandra.koenig@ait.ac.at}, \email{stefan.schauer@ait.ac.at})}
	\and Stefan Schauer\footnotemark[2]
	\and Vincent B\"urgin\thanks{University of Passau, Faculty of Computer Science and Mathematics, Passau, Germany (\email{v.buergin@gmx.de}, \email{jeremias.epperlein@gmail.com}, \email{fabian.wirth@uni-passau.de})}
	\and Jeremias Epperlein\footnotemark[3]
	\and Fabian Wirth\footnotemark[3]
}

\begin{document}

\maketitle

\begin{abstract}
We consider a family of distributions on which natural tail orders can be constructed upon a  representation of a distribution by a (single) hyperreal number. Past research revealed that the ordering can herein strongly depend on the particular model of the hyperreals, specifically the underlying ultrafilter. Hence, our distribution family is constructed to order invariantly of an ultrafilter. Moreover, we prove that it lies dense in the set of all distributions with the (same) compact support, w.r.t. the supremum norm. Overall, this work presents a correction to \cite{rass_game-theoretic_2015, rass_decisions_2016}, in response to recent findings of \cite{burgin_remarks_2021}. 
\end{abstract}



\section{Introduction}
While classical game theory concerning finite games with payoffs as matrices over $\R$ has been deeply studied and applied in many branches of science, applications in security and risk management motivated the study of analogue games taking their payoffs in the \emph{hyperreal} space $\HR$ .
Among the reasons to look at games over $\HR$ is the need to base decisions on objects carrying more information than just a number, and probability distributions are natural candidates here. A large class of practically relevant univariate distributions can be described uniquely in terms of their moment sequence, so that (via a Taylor series expansion of the characteristic function), we can uniquely associate an univariate distribution function $F:\Omega\to\R$ for a random variable $X$ with its moment sequence $m_X(n) := \int_\Omega x^ndF(x)$. In this way, we can model losses in a game by probability distributions, thus including more information than could be encoded in a simple real-valued payoff score, by letting the random variable $X\sim F$ be represented by an infinite sequence of moments that we can interpret as a hyperreal number. It can be shown that the usual construction of matrix games optimizing average gains or losses translates into an optimization of a mixed distribution describing the losses suffered from randomized actions in the game, but now described not as a scalar value, but rather as a whole loss distribution (using the law of total probability) \cite{rass_decisions_2016}.

\textit{Notation}\\
Throughout this work, we will let $\hr{x}$ denote a hyperreal value, as opposed to $x$ denoting a value from $\R$. Consistently with the literature on game theory, we let optimal values in either structure appear as $\opt{\hr{x}}$ or $\opt{x}$, respectively. Vectors and matrices over either field appear in bold.

Since the representation of hyperreals as number sequence is a quotient of the entire set of sequences, here denoted as $\R^\infty$, modulo an ultrafilter $\UF$, the natural ordering of best decisions would come as the $\leq$-ordering of $\HR$, which generally depends on $\UF$. Past work \cite{rass_decisions_2016} has proposed conditions under which this ordering was hoped to be total and independent of $\UF$ by exhibiting the $\leq$-order in $\HR$ as equivalent to a stochastic tail-order on the distributions. Formally, \cite{rass_decisions_2016} proposed the following stochastic order:
\begin{definition}\label{def:hyrim-order}
	Let $\HR=(\R^\infty\slash\UF,\leq)$ be an instance of the hyperreal space using an ultrafilter $\UF$, and with $\leq$ being the induced total order. Let $X,Y$ be univariate real-valued random variables, supported on a (common) compact subset $[a,b]\subset[1,\infty)\subset\R$. We put the two random variables into the order relation $X\preceq Y$ if and only if the hyperreal numbers defined from the respective moment sequences, i.e., $\hr{x}=(m_X(n))_{n\in\N}$ and  $\hr{y}=(m_Y(n))_{n\in\N}$ satisfy $\hr{x}\leq \hr{y}$ within $\HR$.
\end{definition}

The proposal in \cite{rass_decisions_2016} attempted to classify a set of distributions that are totally ordered under $\preceq$, and for which the ordering was independent of the choice of $\UF$. This would have delivered a ``natural'' lift of games with distributions as payoffs into an analogue of classical matrix games, only played inside $\HR$. However, it was found later in \cite{burgin_remarks_2021} that the conditions were too weak to equate the $\leq$-order in $\HR$ to a -- more interpretable -- stochastic tail order. Specifically, the examples of unimodal (even monotone) distributions constructed in \cite{burgin_remarks_2021} exhibit some oscillatory behavior that precluded it from comparing in the same $\leq$-sense under all ultrafilters, thus showing that the dependency on $\UF$ is still there. As a second observation of \cite{burgin_remarks_2021}, the practical computation of Nash equilibria in $\HR$ is more involved, as the usual convergence results known for classical games, upon application to games in $\HR$, may fail (as already recognized earlier in \cite{rass_game-theoretic_2015}). Indeed, the computation of equilibria along known direct or iterative (online-learning) algorithms delivers something that is \emph{not} a Nash equilibrium. This is shown by an instructive counterexample due to \cite{burgin_remarks_2021}, which we will repeat it later as Example \ref{exa:vincents-example}.

The purpose of this note is twofold: first, we give more stringent conditions on the class of distributions (Definition \ref{def:set-f-ab}) upon which we construct games to avoid the unpleasant phenomena reported in \cite{burgin_remarks_2021}. Specifically, we will work with piecewise polynomial densities on compact supports to regain a stochastic tail order that is total and independent of $\UF$ when defined as the $\leq$-order within $\HR=\R^\infty\slash\UF$ (Lemma \ref{lem:ordering-is-natural}). The second purpose is to study the equilibria concept related to the tail ordering, and to demonstrate that the conventional notion of a Nash equilibrium does not directly apply in this setting. Further, we discuss the interpretation of why a player would (not) unilaterally deviate from the equilibrium. This is where Example \ref{exa:vincents-example} demonstrates that a utility maximizer, whose decision is based on more complex objects than numbers, can have an incentive to deviate from an equilibrium if there is more than one goal to optimize. A meaningful notion of equilibrium in this setting is proposed using the concept of a lexicographic Nash equilibrium (Definition \ref{def:lex-order-nasheq}). It turns out that this notion (i) coincides with a conventional Nash equilibrium if the game is a ``standard'' one, i.e., has its payoffs all in $\R$, or equivalently, the distributions are all degenerate. However, it is not a usual Nash equilibrium in $\HR$, as Example \ref{exa:vincents-example} demonstrates, i.e., we introduce a concept that really differs from existing approaches.

From a practical perspective, the concept of a lexicographic Nash equilibrium appears good enough to make decisions as it accounts for explicit goal priorities that are often encountered, and does not require a decision maker to assign numeric weights to reflect importance of goals if the decision making follows a Pareto-optimization technique.
\begin{remark}
In general, a Pareto-optimum is no substitute for a lexicographic optimum
over a set of continuous utility functions $u_1,\ldots,u_d$. To see this,
suppose on the contrary that the lexicographic order $x\lexleq y :\iff
(u_1(x)$, $\ldots,$ $u_d(x))\lexleq (u_1(y),\ldots,u_d(y))$ would be equally
computable as a Pareto-optimum on $\set{x,y}$. This optimization would use a
scalarized function $v = \sum_i w_i u_i$ (with real-valued coefficients
$w_i$) to declare $x\lexleq y$ if and only if $v(x)\leq v(y)$. But since all
$u_i$ are continuous, so would $v$ be as a representation of the
lexicographic order. This, however, contradicts the well known fact that the
lexicographic order does not generally admit a continuous function to
represent it.
\end{remark}	
From the modelling perspective of statistics, the restriction to piecewise polynomial densities comes with a controllable uniform error (Lemma \ref{lem:uniform-density-approximation}), or topologically speaking, the set of such probability densities is dense w.r.t. the $\infnorm{\cdot}$-topology inside the class of all distributions with compact support and continuous density functions (w.r.t. Lebesgue measure).

\section{Matrix Games in the Hyperreal Space}
Consider a normal form game defined by a payoff matrix $\vec A\in\R^{n\times
m}$. We denote the action spaces as $AS_1=\set{1,2,\ldots,n}$ and
$AS_2=\set{1,2,\ldots,m}$, with their corresponding convex hulls (simplexes)
$\simplex(AS_1)=\{(x_1,\ldots,x_n)\in[0,1]^n:$ $\sum_{j=1}^n x_j=1\}$, and
$\simplex(AS_2)$ defined likewise. Nash's theorem implies that there are
elements $(\opt{\vec x},\opt{\vec y})\in\simplex(AS_1)\times\simplex(AS_2)$
such that, when the game is zero-sum, the expected payoff $(\vec x,\vec
y)\mapsto \vec x^T \vec A \vec y$ satisfies
\begin{equation}\label{eqn:saddle-point}
	\vec x^T \vec A \opt{\vec y}\leq \opt{\vec x}^T\vec A\opt{\vec y}\leq(\opt{\vec x})^T\vec A\vec y
\end{equation}
for all $\vec x,\vec y\in \simplex(AS_1)\times\simplex(AS_2)$. By \L os' theorem and the transfer principle \cite{robinson_iterative_1951}, we can (syntactically) rewrite the conditions into all involved variables being hyperreal quantities.

The saddle point condition \eqref{eqn:saddle-point} now reads in hyperreal terms
\[
(\hr{\vec x})^T \hr{\vec A}\opt{\hr{\vec y}}\leq (\opt{\hr{\vec x}})^T\hr{\vec A}\opt{\hr{\vec y}}\leq (\opt{\hr{\vec x}})^T\hr{\vec A}\hr{\vec y}
\]
in which $\hr{\vec A}\in\HR^{n\times m}$ can (among others) now represent moment sequences for probability distributions that describe the payoffs in some real-life game. Since Nash's theorem about the existence of the equilibrium $(\opt{\hr{\vec x}})^T\hr{\vec A}\opt{\hr{\vec y}}$ is a statement in first order logic, it holds analogously inside $\HR$, only at the caveat of the equilibriums strategy being in hyperreal probability terms.

For the decision making practice, reaching a practically useful equilibrium requires a practical interpretation and use for the hyperreal mixed strategies, provided that we can compute them. Since the arithmetic within $\HR$ depends on the ultrafilter $\UF$, we could look for a way of computing an equilibrium that works with only a ``minimum'' amount of arithmetic. Fictitious play is one candidate, which iteratively converges (under various conditions, including the game to be zero-sum such as implied by \eqref{eqn:saddle-point}) to a Nash equilibrium. In simplified terms, the procedure is that both players keep records of their opponents actions, defining empirical mixed strategy estimates $(\vec x_n)_{n\in\N}$ and $(\vec y_n)_{n\in\N}$, and in the next round $n+1$ reply best to either $\vec A\cdot \vec y_n$ for player 1, or reply best to $\vec x_n^T\cdot \vec A$ for player 2. While it is known that this process converges as $n\to\infty$ \cite{robinson_iterative_1951} to an equilibrium, the transfer to the hyperreals translates this convergence into one of $\hr{n}\to\infty$ over a sequence of hyper-integers in ${^*\N}=\N^\infty\slash\UF$. Such sequences are considerably longer than divergent ones within $\N$, since for convergence inside $\HR$, we need to reach beyond integer infinity. Nonetheless, fictitious play carried out by both players looking for best replies by deciding the $\leq$-order in $\HR$ does carry to convergence over integer sequences $n\to\infty$ in $\N$ under the small tweak of accepting two payoffs as identical up to a numeric roundoff error $<\eps$, when $\eps$ is an a priori fixed (machine-)precision (as implemented in older versions \texttt{1.x} of \cite{hyrim_2020}). However, the result is not a Nash equilibrium, as Example \ref{exa:vincents-example} demonstrates.

The observation to take away from this, however, is that a sequence that converges in $\HR$ can admit a subsequence that converges within the subset $\R$ for a sequence of indices within $\N$. This is in some contrast to the usual fact that a convergent sequence has all its subsequences convergent to the same limit. The lesson for practice is thereby to strictly distinguish convergence that we can compute, namely over sequences in $\N$ from theoretical convergence in a superset of the integers, such as $\HR$.

Since Nash equilibria do exist in $\HR$ but are literally out of reach via iteration, it pays to study the object that we can find by online learning and (classical) optimization techniques. We believe that this new notion of equilibrium, coined a \emph{lexicographic equilibrium} in Definition \ref{def:lex-order-nasheq}, may provide new possibilities to address some old and common criticism of conventional Nash equilibria (bounded rationality) in Section \ref{sec:discussion}.

The proposed way to escape the problems discovered by \cite{burgin_remarks_2021} is to restrict our attention to density functions that are piecewise polynomial, and therefore "sufficiently well behaving" for our purposes of decision making and game theory. We develop this idea over the next series of sections, starting with Section \ref{sec:piecewise-polynomial-densities}, pointing out this restriction as rather mild.

\section{Piecewise Polynomial Approximations of Probability Densities}\label{sec:piecewise-polynomial-densities}

\begin{lemma}\label{lem:uniform-density-approximation}
	Let $f:[a,b]\to\R$ be a continous probability density function supported on the compact interval $[a,b]\subset\R$. Then, for every $\eps>0$, there is a piecewise polynomial probability density $g$ that uniformly approximates $f$ as $\infnorm{f-g}<\eps$.
\end{lemma}
\begin{proof}
	Pick any $\delta>0$. It is straightforward to use Weierstra\ss' approximation theorem to get a polynomial $p$ that uniformly $\delta$-approximates $f$ on the given interval. The issue is that (i) $p$ may take on negative values, and (ii), $p$ is not necessarily normalized to be a probability distribution. To fix both, we define the sought function $g(x):=\alpha\cdot p^+(x)$ with $p^+(x)=\max(0,f(x))$ for $x\in[a,b]$, with a normalization constant $\alpha>0$ chosen to make $\int_a^b g(t)dt=1$. Let us postpone the role of $\alpha$ until a little later, and look at how well the function $p^+$ approximates $f$.
	
	By definition, $p^+$ is different from $p$ only at positions $x\in[a,b]$ when $p(x)<0$, and otherwise identical. Since $f$ is bounded from below by zero, $p^+$ can only get ``closer'' to the graph of $f$, and hence also satisfies $\infnorm{f-p^+}<\delta$. We will make use of this later. Now, let us normalize $p^+$ into a probability density, i.e., choose $\alpha>0$ such that $\int_a^b \alpha\cdot p^+(t)dt=1$, and look at the maximal error $\infnorm{f-\alpha\cdot p^+}$ on the interval $[a,b]$. We apply the triangle inequality twice after expanding the inner difference,
	\begin{align*}
		\infnorm{f-\alpha\cdot p^+} &= \infnorm{f-\alpha\cdot p^+-p+p+\alpha\cdot p-\alpha\cdot p} \\
		&= \infnorm{(f-p) + (\alpha p-\alpha p^+) + (p-\alpha p)}\\
		&\leq \infnorm{f-p} + \abs{\alpha}\cdot\infnorm{p-p^+} + \infnorm{p-\alpha p}.
	\end{align*}
	Therein, the first term is $<\delta$ by construction (Weierstra\ss' theorem). The second term is the maximum difference between $p$ and $p^+$, which is also bounded by $\delta$ since $p$ cannot fall below $-\delta$ as $f$ is bounded to be $\geq 0$, and $p$ is ``$\delta$-bound'' to $f$.
	
	The third term $p-\alpha p$ attains a maximum where its first order derivative $p'-\alpha\cdot p'=(1-\alpha)p'$ vanishes. Now, if $\alpha=1$, then $p^+$ is normalized already and we are done. Otherwise, $\alpha\neq 1$, and we can just look for a maximum of $p$. Again, since $p$ is bound to a deviation from $f$ that is everywhere smaller than  $\delta$, its maximum must be in an $\delta$-neighbourhood of the maximum of $f$, and we can approximately locate the extreme value $p_{\max}\in [\infnorm{f}-\delta,\infnorm{f}+\delta]$. Wherever this maximum is attained, the same position $x=x_{\max}$ will also maximize the deviation $\alpha\cdot p(x)$. Thus, the maximum possible deviation between $p_{\max}$ and $\alpha\cdot p_{\max}$ is $\leq \abs{\alpha(\infnorm{f}+\delta)-(\infnorm{f}-\delta)}\leq \abs{\alpha\cdot\infnorm{f}-\infnorm{f}}+\abs{\alpha\delta-\delta}=\abs{\alpha-1}\cdot(\infnorm{f}+\delta)$.
	
	Combining all three bounds, we find
	\begin{equation}\label{eqn:uniform-bound}
		\infnorm{f-\alpha\cdot p^+}\leq\delta+\abs{\alpha}\cdot\delta+\abs{\alpha-1}\cdot(\infnorm{f}+\delta)
	\end{equation}
	
	To exhibit this bound to become arbitrarily small ultimately, let us finally estimate the value $\alpha$. To this end, let us return to our previous observation that $f-\delta\leq p^+\leq f+\delta$ due to the uniform approximation property. Integrating the inequalities from $a$ to $b$, we find
	\[
	1-\delta(b-a)\leq \int_a^b p^+(t)dt=\frac 1{\alpha}\leq 1+\delta(b-a)
	\]
	Taking the reciprocal gives us
	\[
	\frac 1{1+\delta(b-a)}\leq \alpha\leq \frac 1{1-\delta(b-a)}
	\]
	Letting $\delta\to 0$, will make $\alpha\to 1$ (sandwich theorem), and thereby also lets the bound \eqref{eqn:uniform-bound} become $<\eps$ ultimately (for any $\eps>0$ that we can choose in advance), thus proving the claim.
\end{proof}

From Lemma \ref{lem:uniform-density-approximation}, we can state that without loss of too much generality, we may approximate any loss distribution of arbitrary shape, yet compactly supported within $[1,\infty)$ by a replacement distribution that is piecewise polynomial. More concisely said, the set of piecewise polynomial densities is even dense inside the entire set of probability distributions that are compactly supported and with a continuous density function (i.e., absolutely continuous w.r.t. the Lebesgue measure).

\section{Approximating Losses Piecewise Polynomially}
Without the uniform error bound, a piecewise polynomial approximation is simple,e.g. by linear interpolating between a chosen set of points within the support $[a,b]$, or by using polynomial splines. For distributions that do not show ``wild oscillating behaviour'', such a simpler approximation could be handy as well.

\begin{remark}[Working with Approximations in Practice] \label{rem:practical-approximations}
	While one may argue that any such approximation (including an $\eps$-uni\-form one) may invalidate the probability model derived to describe the loss, the relevance for practice of security risk management appears only mildly affected, if not unaffected at all. After all, risk is not a physical quantity to enjoy known background dynamics that would lend themselves to the derivation of an ``exact'' probabilistic model, and is in most practical instances a matter of subjective modeling out of experience and domain expertise, and actuarial science \cite{klugman_loss_1998}. In this view, and given that loss data is hard to reliably estimate or predict generally, the approximation comes essentially to the replacement of one approximation of reality by yet another approximation, and possibly so with only an arbitrarily small additional error bounded by $\eps$.
\end{remark}

The approximation actually has considerable technical advantages, since it ensures the totality of the tail order by results of \cite{burgin_remarks_2021}\footnote{Intuitively, the examples constructed by \cite{burgin_remarks_2021} no longer apply, since the two functions $f$ and $g$ can only have finitely many oscillations (since both have a finite degree) around each other.
}, since any two densities $f,g$ that are both piecewise polynomial will eventually dominate one another in a right neighborhood $(b-\delta,b)$ within the support $[a,b]$. More importantly, this also makes the tail order based on moment sequences (see Definition \ref{def:hyrim-order}) ``natural'' as being independent of the ultrafilter underneath the hyperreal space.

\begin{lemma}\label{lem:ordering-is-natural}
	Let $[a,b]\subset[1,\infty)$ being given, and let $\F$ be the set of continuous probability densities supported on $[a,b]$ that are piecewise polynomial. This set is totally ordered under $\preceq$ induced by the natural ordering on the hyperreal space, 
	and the ordering is independent of the ultrafilter $\UF$ therein.
\end{lemma}
\begin{proof}
	Let two densities $f,g$ be given, defined by individual polynomials on partitions of $[a,b]$ given by $a=t_1<t_2<\ldots<t_n=b$ for the density $f$, and $a<s_1<s_2<\ldots<s_m=b$ for the density $g$. Then, the difference $f-g$ is again piecewise polynomial on a finer partition $a=r_1<r_2<\ldots<r_\ell=b$, where we can choose the points $r_i$ such that in each open subinterval $I_j=(r_j,r_{j+1})$, we have the trichotomy of either $f|_{I_j}=g|_{I_j}$, or $f|_{I_j}<g|_{I_j}$ or $f|_{I_j}>g|_{I_j}$, where the "or" is exclusive and the functions are understood as restricted to the subinterval. Note that this partition is necessarily finite, since all polynomials have a finite degree (and thus cannot oscillate infinitely often). It then follows by \cite[Lemma 8 and Proposition 9]{burgin_remarks_2021} (generalizing and extending prior incomplete arguments about this, given in \cite{rass_decisions_2016}), that the dominance relation between $f$ and $g$ on the last interval $[r_{\ell-1},r_{\ell}]$ determines which moment sequence (that of $f$ or that of $g$) diverges faster. Specifically, if we write $(m_f(n))=(E_f(X^n))_{n\in\N}$ for the moment sequence of $f$, let $(m_g(n))$ be the likewise defined moment-sequence for the density $g$, then the ultimate dominance of $g$ over $f$ (assumed here without loss of generality, for otherwise, we may just switch names between $f$ and $g$), that (see \cite[Proposition 9]{burgin_remarks_2021}) there is an index $N\in\N$ such that the moments satisfy $m_f(k)<m_g(k)$ for all $k\geq N$. This puts the two sequences into a $\leq$-relation within the hyperreals, and this order is in a way ``natural'' as it is independent of the ultrafilter $\UF$: to see this, note that the order is only violated on a finite subset of indices in $\set{1,\ldots,N-1}$, whose complement must, by definition, be a member of $\UF$.
\end{proof}

From the assumption of piecewise polynomial densities (supported by Lemma \ref{lem:uniform-density-approximation}, and the canonicity of the ordering as follows from Lemma \ref{lem:ordering-is-natural}, we get a constructive  criterion to decide the $\preceq$-relation:
\begin{proposition}\label{prop:derivative-criterion}
	Let $\F$ be as in Lemma \ref{lem:ordering-is-natural}. Then, for every $f$ we can calculate a finite-dimensional vector $\vec v_f\in\R^n$, where $n$ depends on $f$, with the following property: given two density functions $f,g$ with computed vectors $\vec v_f=(v_{f,1},\ldots,v_{f,n})$ and $\vec v_g=(v_{g,1},\ldots,v_{g,m})$, we have $f\prec g$ if and only if $\vec v_f\lexlt \vec v_g$, taking absent coordinates to be zero when $n\neq m$. If the two vectors are lexicographically equal, then $f=g$.
\end{proposition}
\begin{proof}
	The first part of this result is literally taken from \cite{rass_game-theoretic_2015} and repeated here only for convenience of the reader. Let us take the partitioning $a=r_1<r_2\ldots r_{\ell}=b$ from the proof of Lemma \ref{lem:ordering-is-natural}, and look for which density is below the other in the last interval $I_{\ell}=[r_{\ell-1},b]$ only. To ease our notation, let us for the moment restrict $f$ and $g$ only to the interval $[r_{\ell-1},b]$, and let $f,g$ synonymously mean $f|_{I_\ell}$ and $g|_{I_\ell}$.
	
	Let us take a ``mirrored'' view on the functions around the vertical line at
	$x=b$ and shift the functions to the left by a substitution $x\gets x-b$, so that the interval of interest is now $[0,b-r_{\ell-1}]$ for the replacement functions $f(x)\gets f(-(x-b))$ and $g(x)\gets g(-(x-b))$. Clearly, whichever function grows slower in a neighborhood $[0,\eps)$ with $\eps>0$ is the $\preceq$-lower function. Deciding this is easy by looking at $k$-th order derivatives at $x=0$: we will inductively show that if
	\begin{equation}\label{eqn:derivative-criterion}
		((-1)^k\cdot f^{(k)}(0))_{k\in\N}<_{lex} ((-1)^k\cdot g^{(k)}(0))_{k\in\N},
	\end{equation}
	then $f\preceq g$.
	
	For $k=0$, if $f(0)<g(0)$, then $f\preceq g$, since the continuity implies that the relation holds in an entire neighborhood $[0,\eps)$ for some $\eps>0$. This completes the induction start.
	
	For the induction step, assume that $f^{(i)}(0)=g^{(i)}(0)$ for all $i<k$,
	$f^{(k)}(0)<g^{(k)}(0)$, and that there is some $\eps>0$ so that
	$f^{(k)}(x)<g^{(k)}(x)$ is satisfied for all $0\leq x<\eps$. Take any such
	$x$ and observe that
	\begin{align*}
		0 &> \int_0^x\left(f^{(k)}(t)-g^{(k)}(t)\right)dt =f^{(k-1)}(x)-f^{(k-1)}(0)-\left[g^{(k-1)}(x)-g^{(k-1)}(0)\right]\\
		&=f^{(k-1)}(x)-g^{(k-1)}(x),
	\end{align*}
	since $f^{(k-1)}(0)=g^{(k-1)}(0)$ by the induction hypothesis. Thus,
	$f^{(k-1)}(x)<g^{(k-1)}(x)$, and we can repeat the argument until $k=0$ to
	conclude that $f(x)<g(x)$ for all $x\in[0,\eps)$.
	
	For returning to the original problem, we must only revert our so-far
	mirrored view by considering $f(-x), g(-x)$ in the above argument. The
	derivatives accordingly change into $\frac {d^k}{dx^k}f(-x)=(-1)^k
	f^{(k)}(x)$, to arrive at criterion \eqref{eqn:derivative-criterion}.
	
	Now, observe that the functions $f,g$ were all piecewise polynomial, and especially are so on the original interval $[r_{\ell-1}, r_{\ell}]$. Since the two have finite degree, both are in $C^\infty$, with the derivative sequences eventually becoming and remaining zero after $1+\deg(f)$, resp. $1+\deg(g)$ derivations.
	
	The vectors $\vec v_f, \vec v_g$ are just defined to operationalize condition \eqref{eqn:derivative-criterion} by collecting the alternated-sign derivatives up to order $1+\max(\deg(f),\deg(g))$ as in \eqref{eqn:derivative-criterion}. The polynomial of lower degree will naturally have less nonzero derivatives, but by just carrying on the derivation on the zeroes, will thus only become extended with zeroes until it has the same dimension as the vector for the polynomial with the larger degree.
	
	If all derivatives of orders $0,1,2,\ldots,1+\max\set{\deg(f),\deg(g)}$ coincide, equivalently, if the vectors $\vec v_f$ and $\vec v_g$ are identical, then the polynomials are themselves identical (note that the ``zero-th'' derivative is explicitly needed within the vectors).
	
	If the two densities coincide on the subinterval $[r_{\ell-1},r_{\ell}]$, then the order is determined by which function dominates the other on the next subinterval $[r_{\ell-2},r_{\ell-1}]$. Analogously, we get another sequence of derivatives for the polynomials defining the density on this subinterval, and we can just ``append'' them to the so-far constructed vectors $\vec v_f,\vec v_g$. Since the degree of the polynomial on this subinterval is again finite, the vectors also remain finite.
	
	We can then repeat this procedure up to the last interval $[a=r_1,r_2]$, to either decide the order there lexicographically on the so-far constructed vectors $\vec v_f,\vec v_g$, or conclude that the densities are identical.
\end{proof}

\begin{remark}
	We emphasize that along all these lines, we strongly rely on the finiteness of the support, partitions and also degrees of polynomials. None of the above results may hold after dropping any of these finiteness assumptions, as counterexamples in \cite{burgin_remarks_2021} demonstrate.
\end{remark}

Picking up on remark \ref{rem:practical-approximations}, it is also important to bear in mind that the uniform approximation of Lemma \ref{lem:uniform-density-approximation} does not naturally extend to being also an approximation of derivatives. Practically, one may compute a uniform polynomially approximation in several ways, such as Remez' algorithm or using Bernstein polynomials. M-splines \cite{ramsay_monotone_1988} offer an appealing alternative in having the properties of probability densities by construction, at the cost of no longer necessarily providing a uniformly good approximation. A possibility to simultaneously $\infnorm{\cdot}$-approximate all derivatives up to a fixed order is offered by the direct method that Weierstra\ss \ used to prove the approximation theorem, namely by convolution with a truncated polynomial (of proper choice to approximate a Gaussian kernel), or also by using splines with a Bernstein polynomial base. Practically, one may consider taking ``approximate equalities'' of quantities within a deviation of $<\eps$, i.e., the chosen approximation accuracy. Imposing this rounding on all derivatives, if appropriate for the application, may settle the issue with the approximation in the easiest way.

\subsection{Experimental Evaluation}\label{sec:experiments}
It always pays to run some numeric experiments to verify theoretical claims and assess the practical usefulness of results. In our case of the uniform piecewise polynomial approximation (Lemma \ref{lem:uniform-density-approximation}), we took some artificial data to compile a kernel density distribution estimate, and applied the techniques of Lemma \ref{lem:uniform-density-approximation} to it by seeking a Bernstein polynomial approximation to the kernel density (via the \texttt{splines2} package for R \cite{splines2-package}), truncating it at regions below zero, and renormalizing to unit integral in each trial. Towards reaching the desired accuracy, we set $\eps=0.1$ as a (crude) accuracy bar, and ran an exponential search by doubling the order $d$ of the Bernstein polynomial until the desired accuracy was undercut, leaving the ``optimal'' order somewhere in the interval $[d,2d]$. Within this remaining search space, we ran a binary search (letting the uniform approximation error again increase) until we first exceed the threshold $\eps=0.1$. Overall, the best approximation was found at order 2456 of the Bernstein polynomial, with a uniform error of $\leq 0.09998103$. This, together with the crude approximation target of 0.1, shows that the convergence \emph{can}, in general, be rather slow. The implementation was done in R \cite{R-core}, version 4.1.0.

\subsection{(Un)ambiguity of the Order}
We emphasize that the proximum given by Lemma \ref{lem:uniform-density-approximation} is generally non-unique, which can induce ambiguities in the ordering. Thus, and the construction does not lend itself to a ``canonic’’ ordering of \emph{all} distributions that would independent of the ultrafilter $\UF$. If the ordering ought to be independent of $\UF$ by restricting it to $\F$, it will depend on which particular candidate members of $\F$ approximate the distributions in question. Otherwise, resorting to the ordering of hyperreals without restricting the set of density functions (other than having a hyperreal representative), the order will depend on the particular $\UF$. One way to escape the issue is via discretization: we can finitely partition the interval $[a,b]$ and assign the respective mass that two distributions $F_X,F_Y$ put on the subintervals as probability values, to get categorical distributions $\hat F_X,\hat F_Y$ that naturally order under $\preceq$ and do so independently of $\UF$ (for example, histograms constructed from a finite lot of empirical data will naturally deliver this). This is in fact consistent with the contemporary recommendations of quantitative risk management \cite{munch_wege_2012}, as it can avoid numerical accuracy issues (of several kinds, not only robustness). In addition, it has the appeal of allowing for an account of subjective risk appetite, meaning that people who are risk averse or risk seekers, can define the partitioning of the interval accordingly to the ranges that most strongly govern their decision making under risk. This idea has previously been formalized by \cite{alshawish_risk-based_2020}, using the convexity or concavity of subjective utility functions to define the partitioning via quantiles.

\section{Game Theory over Tail Orders}
\begin{definition}[the set $\F$]\label{def:set-f-ab}
	Let $1\leq a<b<\infty$ be two real numbers, from which we define the set $\F$ to contain all probability distributions that (i) are absolutely continuous w.r.t. the Lebesgue measure, and (ii) have a density that is piecewise polynomial over a finite partition of the compact interval $[a,b]$.
\end{definition}

With the so-restricted set $\F$, let us set up a matrix-game with payoffs only being piecewise polynomial densities from $\F$, then these are totally ordered, and their hyperreal representatives re-create the game as a humble matrix game entirely within the hyperreal space. Rigorously, let the game matrix $\vec A\in\F^{n\times m}$ be defined over the finite strategy spaces $AS_1=\set{1,\ldots,n}, AS_2=\set{1,\ldots,m}$, and consider the corresponding mixed strategies $\vec x,\vec y\in\simplex(AS_1)\times\simplex(AS_2)\subset\HR^2$ explicitly as categorical distributions with hyperreal probability masses. Thus, letting the whole game be played within the hyperreals, we have all necessary ingredients (continuity of the expected payoff $\vec x\cdot\vec A\cdot\vec y$ w.r.t. the order topology, and Glicksberg's theorem to assure the existence of equilibria) ready to get a whole theory of games within the hyperreals ``for free''.

The practical caveat comes in if we speak about playing games entirely within $\HR$, in which case the mixed strategies themselves become vectors of hyperreal numbers. These ``are'' in a way probabilities, but are far less trivial to interpret in a frequentistic way. Whether there is an alternative interpretation in a more subjective manner like in Bayesian statistics, is a question left open here.

Generally, without being restricted on $\F$, the $=$-relation over the hyperreals induces only a partitioning of all distributions into equivalence classes, but the equality of the hyperreal representatives does not imply an identity of the corresponding distributions. To this end, we require suitable restrictions of the set of distributions, such as to $\F$.

While sounding technically inconvenient, it may offer an interesting explanation for bounded rationality effects, whose investigation is -- in our view -- a matter outside purely mathematical considerations and thus left as a pointer of possible research. Stated more concisely:

\begin{quote}
	\emph{If the utility-maximizing paradigm is violated in practical situations for real-valued utility functions, can a seeming deviation from an equilibrium be nonetheless rational if the utility maximization is just done in a different structure than $\R$?}
\end{quote}

Our work is exactly an positive instance of the above question:
The idea of replacing real values by distributions for the sake of a ``more informed'' decision making exhibits this effect, since if unrestricted, the ordering of distributions represented by moment sequences making up hyperreal numbers would depend on the specific ultrafilter. The above question is then linked to whether effects of bounded rationality can be explained as rational under the utility maximization paradigm, only using a properly chosen ultrafilter. We leave this question unanswered here, since it appears to run deeper than the scope of this work.

This effect can be demonstrated even without resorting to any hyperreal arithmetic or games, if one seeks to leverage Proposition \ref{prop:derivative-criterion} for the purpose of playing games over lexicographic order, with the hyperreal machinery as a mathematical backup. Suppose we would play a matrix game with distributions as payoffs, i.e., with a matrix $\vec A\in\F^{n\times m}$, and that we run fictitious play as an online learning algorithm for both players to converge towards a Nash equilibrium alternatingly responding $\preceq$-optimal to the other player, using Proposition \ref{prop:derivative-criterion}. It is well known that if fictitious play converges, then the limit is an equilibrium. However, the next example shows that the result is not necessarily a Nash equilibrium:

\begin{example}[\cite{burgin_remarks_2021}]\label{exa:vincents-example}
	Consider a $2\times 2$ zero sum game composed from categorical distributions on the common support $\Omega=\set{1,2,3}$, given by the following payoff structure, with the lexicographically maximizing row player having strategies $r_1,r_2$ and the lex-minimizing column player having strategies $c_1,c_2$. The lexicographic order is herein taken from right to left.
	
	\begin{center}
		\vspace{0.5\baselineskip}
		\centering
		\begin{tabular}{|c|c|c|}
			\hline
			$\vec A$ & $c_1$ & $c_2$ \\
			\hline
			$r_1$ & (0.3, 0.2, 0.5) & (0.6, 0.3, 0.1) \\
			\hline
			$r_2$ & (0.8, 0.1, 0.1) & (0.3, 0.2, 0.5) \\
			\hline
		\end{tabular}\vspace{0.5\baselineskip}
	\end{center}
	
	The point observed in \cite{burgin_remarks_2021} is that a Nash equilibrium $(\opt{\vec x},\opt{\vec y})$ in this game necessarily is also an equilibrium in the game $G_3$ composed only from the third coordinates, i.e., $G_3=\binom{0.5~0.1}{0.1~0.5}$. This game has a unique Nash equilibrium (from classical calculations) coming to $\opt{\vec x}=\opt{\vec y}=(0.5, 0.5)$ for  and giving the average payoff $\opt{\vec x}\cdot\vec A\cdot \opt{\vec y}=(0.5, 0.2, 0.3)$. However, player 1 can unilaterally deviate to lexicographically gain more by playing $\vec x'=(1,0)$ to receive $(1,0)^T\cdot\vec  A\cdot \opt{\vec y}=(0.35, 0.25, 0.3)$.  Thus, the strategy $(0.5, 0.5)$ is not optimal for the row-player, and since there is no other equilibrium possible in $G_3$, \cite{burgin_remarks_2021} concluded that there is no Nash equilibrium at all in this game, w.r.t. lexicographic order.
\end{example}

The way to resolve the apparent paradox is to reconsider the notion of an equilibrium from a different angle, and particularly bearing in mind that the players actually engage in more than one game simultaneously, which changes the incentive mechanisms in a crucial way.

\section{When Unilateral Deviations from an Equilibrium can be Rational}

The phenomenon of seemingly rational unilateral deviation can be attributed to an implicit yet flawed subsequent assumption, seemingly ``implied'' by assuming that the other player follows the equilibrium \emph{after} a player has deviated, which may just not happen in reality: recall, just intuitively, that a Nash equilibrium is understood as a strategy profile in which a player, \emph{assuming} that all its opponents follow the equilibrium, has no incentive to deviate from its own equilibrium strategy. Let us, from player 1's perspective, put this assumption to question: The implicit error is made when we assume that the opponents would not react on player 1's deviation, which will not happen in reality whenever the game is repeated. In fact, any strategy that player 1's opponents may constantly play (whether mixed or not) may open a door for player 1 to increase its own revenue, but only as long as the opponents do not likewise respond to similar opportunities for themselves. This effectively initiates an online learning process, yet not necessarily equal to fictitious play (for reasons outlined above and corroborated by Example \ref{exa:vincents-example}).

To see the effect, take a simple diagonal game with payoff structure being the identity matrix
\[
\vec A = \begin{pmatrix}
	1 & 0 & 0 \\
	0 & 1 & 0 \\
	0 & 0 & 1
\end{pmatrix}
\]
It is immediate that this game has a unique Nash equilibrium being $\opt{\vec x}=\opt{\vec y}=(1/3,1/3,1/3)$ for both players, paying $v=1/3$ as the saddle point value.

Now, since this is an equilibrium, player 1 cannot gain anything more in \emph{this} game by playing different to $\opt{\vec x}$. But what if it could gain more in a second game played simultaneously, while knowing that it can safely deviate as long as player 2 sticks to $\opt{\vec y}$? This is entirely legitimate and covered by the assumptions underneath an equilibrium, since it does not speak about players engaging in \emph{several} competitions. But this is the situation that we have, and it can destabilize the equilibrium.

The problem kicks in when player 2 is also engaged in the same ``second'' game, where player 1 seeks to improve its payoff. If so, then player 1's deviation relying on the equilibrium property of one game may create an incentive for the other player to deviate too, simply because there is a second game that the players also adapt to. The problem reported in Example \ref{exa:vincents-example} is merely because player 1's inventive to deviate is not to win more in \emph{this game}, but rather to gain more in \emph{another} game that it plays simultaneously, namely game $G_2$. Clearly, the equilibrium in $G_3$ is not also an equilibrium in $G_2$, so there \emph{is} room for improvement. And this is the humble reason why player 1 can look for an alternative strategy to win more in $G_2$, while the payoff in $G_3$ remains constant, assuming that the opponent sticks with the equilibrium behavior in $G_3$.

This is precisely the point where practical events will unfold into a learning process, since a rational opponent will most likely adapt to the changed situation that player 1 just created. This reaction can trigger a reconsideration of player 1's choice, and so on. Eventually, the two players will enter a ficitious play like process, whose convergence is generally nontrivial.

But this effect is indeed not in contrast to the existence of equilibria at all, and their existence does not rely on whether the mutual learning carries to convergence. The point is to reconsider game $G_2$ with a changed strategy set, according to the optimal behavior in game $G_3$. Indeed, the proper way of finding an equilibrium starts with the computation of one in $G_3$, but after that, we are bound to play only $G_3$-equilibria when entering $G_2$, and not its original pure strategies.

\section{Lexicographic Nash Equilibria}
So, to restore the useful notion of an equilibrium, we propose an extension of the concept to several games. Commonly, this is done by rephrasing optimality in one dimension by Pareto-optimality in several dimensions, but this would be too weak for our purposes, since Proposition \ref{prop:derivative-criterion} induces a strict preference order on the payoff dimensions. This has indeed also practical roots, since several goals in a game, especially in the security context, may come in a clear order of importance. Assigning weights to them for the computation of Pareto-optima is generally less trivial and a more involved task for a practitioner.

\begin{definition}[Lexicographic Nash equilibrium]\label{def:lex-order-nasheq} Let $\vec A_1,\ldots,\vec A_d\in \R^{n\times m}$ be a finite collection of 2-player games, all over the same strategy spaces $AS_0, AS_1$ for all players, and listed in descending order of (lexicographic) importance. We call a strategy profile $(\opt{\vec x_0},\opt{\vec x_1})\in\simplex(AS_0)\times\simplex(AS_1)$ a \emph{lexicographic Nash equilibrium in mixed strategies}, if for any player $i\in\set{0,1}$ upon an unilateral deviation towards ${\vec x}_i'\neq \opt{\vec x}_i$ to improve its revenue in the $k$-th game $(1\leq k\leq d)$, there is an index $k'<k$ and a strategy $\vec x_{1-i}\neq\opt{\vec x}_{1-i}$ such that player $i$'s payoff in game $\vec A_{k'}$ gets worse when the second player also deviates to the joint profile $(\vec x_i',\vec x_{1-i})$.
\end{definition}

Definition \ref{def:lex-order-nasheq} differs only slightly from the usual definition of a Pareto-Nash equilibrium \cite{lozovanu_multiobjective_2005} for multiobjective games, essentially by implying that a deviation from the optimum will indirectly cause losses for the deviating player in regards of a more important payoff dimension than where the improvement was attempted.

Example \ref{exa:vincents-example} in light of this now becomes an illustration of the effect, since player 1, striving to improve in the second coordinate (game $G_2$) thereby incentivizes the opponent to decrease player 1's payoff in the more important game $G_3$ by playing $(0,1)$ over the columns to decrease the reward for player 1 from 0.3 down to 0.1.

For two-player games, the existence of lexicographic Nash equilibria is not difficult to show, and indeed constructive: the software implementation to handle games over tail orders from Definition \ref{def:hyrim-order}, the HyRiM package (as of version \texttt{2.0} \cite{hyrim_2020}), implements lexicographic optimization to this end: given a sequence of game matrices $\vec A_1,\vec A_2,\ldots,\vec A_d$: put $i\gets 1$, and set up a linear program $(LP)$ to compute a saddle point \cite{gibbons_primer_1992}. Then,
\begin{enumerate}
	\item increase $i\gets i+1$, and extend $(LP)$ by the constraint that any strategy played in $\vec A_i$ (the next game in the lexicographic order) to reward the player with at least the saddle point value $v_{i-1}$ for (the previous game) $A_{i-1}$ (or ``at most $\vec v_{i-1}$'' if the player is minimizing). This is to assure that:
	\begin{itemize}
		\item the player can proceed by optimizing the payoff in the next game,
		\item but without worsening its payoff in the previous game.
	\end{itemize}
	\item repeat from step 1, letting $(LP)$ grow one additional constraint in each iteration, until the set of optima has become singleton, or we arrive at $i=d$. The resulting set of optima is then optimal for the overall lexicographic sequence of games.
	
\end{enumerate}
So, is what we get from this procedure a Nash equilibrium? Based on example \ref{exa:vincents-example}, the answer is negative, although it obviously delivers a Nash equilibrium in the case of only a single game being played with one goal (as the computation terminates after the first step in which a conventional Nash equilibrium is computed). Generally, the improvement that player 1 can do destabilizes the situation as it induces an incentive for player 2 to deviate as well, eventually enforcing player 1 to reconsider the deviation. Doing so, it ends up with finding that there is really no incentive to deviate from the equilibrium in $G_3$, since upon a repetition of the game, player 1 would suffer a decrease in the upcoming repetitions.

\section{Discussion}\label{sec:discussion}

The transfer of a game from the real in the hyperreal space is not only a problem of pure theoretical interest, but has its practical application in risk management as it naturally induces a tail ordering, with the rich structure of $\HR$ giving us the fundamental facts about games in $\HR$ almost for free. The price paid for this shortcut is the practical difficulty to translate hyperreal results into real-valued counterparts without sacrificing their properties. We pose this conversion as an open problem for research (also in practical risk management and decision making).

Using hyperreal orderings to define stochastic orders is generally tricky, since to avoid ambiguous orderings, we either need to discretize distributions, model the data using piecewise polynomial distributions in first place already, or choose nonparametric densities with sufficient smoothness to leverage condition \eqref{eqn:derivative-criterion} for an algorithmic decision of $\preceq$. Gaussian kernel density estimates are an example of such an admissible family, as they also avoid oscillations in the tail region, but may become numerically inconvenient when the decision point is remote from the region where most data points are located. Generally, the advice for practical matters of decision making under stochastic orders is thus to either (i) discretize the distribution, making the resulting categories depend on the application context (to define the relevant loss regions), and also (subjective) risk attitudes, or (ii) construct the stochastic models directly from members in $\F$, knowing (from Lemma \ref{lem:uniform-density-approximation}) that these are dense in the larger set of models including the ones appropriate for the application (e.g., from actuarial science or others).

\subsection{Explanations of Bounded Rationality}

Example \ref{exa:vincents-example} demonstrates that a lexicographic Nash equilibrium is not necessarily also a (normal) Nash equilibrium, and hence no generalization thereof. However, it may, again posed as a second question of interdiciplinary study (e.g., involving psychology and cognitive  science in general \cite{starmer_developments_2000}), that bounded rationality effects could also root in the consideration of goals that are not explicitly modelled inside a given game. That is, when a game model is practically found inaccurate since players behave different to what the utility maximizing paradigm would imply, then it could well be for several reasons that we came across during this study. Among them:
\begin{itemize}
	\item the decision being indeed rational, but w.r.t. maximization of utilities that are simply not real-valued, but more complex (e.g., hyperreal in our instance),
	\item the decision is made w.r.t. (perhaps not explicitly known) goal priorities, in which case we may have a lexicographic Nash equilibrium that -- in general -- is not a conventional Nash equilibrium and hence may look like bounded rationality.
\end{itemize}

A notable application of such lexicographic preference decision making was proposed in \cite{alshawish_risk-based_2020}, where the modeling of payoff distributions is used to define a vector of payoffs and sequence of games to reflect the importance of different loss regions for the decision making. Essentially, this concept picks up the idea that gains or losses of certain magnitudes are more important than gains/losses of other magnitudes. For example, a company management may care less about losses of around 10\$, but -- depending on the size of the company -- may care much more about losses in the range of thousands or millions of \$. The elegance of partitioning the range of losses and gains according the importance of different regions is due to this partitioning being possible in account for the individual and subjective risk attitude. For example, a risk averse person may work with a certain (small) range $I_{large}$ of large losses, contrary to a risk seeker, who may be indifferent on a wider interval of large losses (thus taking a higher risk to lose or gain more). The exact definition of these intervals was proposed to come from quantiles computed from the respective convex, concave or linear utility functions to express people's risk attitudes \cite{hillson_understanding_2007}.

\subsection{Connecting Game- and Prospect Theory}
Playing games over tail orders has another possible application to deal with prospects, since those are essentially categorical distributions, and it is long known (and eloquently discussed in  \cite{starmer_developments_2000}) that utility-maximization fails in many practical instances. Lexicographic Nash equilibria are one proposal to study as a possible explanation, solution concept here.

\section*{Acknowledgments}
We would like to acknowledge the help of Tatjana \v{C}obit and Peter Occil in the verification of Lemma \ref{lem:uniform-density-approximation}.

\bibliographystyle{siamplain}


\end{document}